\documentclass[11pt,twoside]{article}
\usepackage{titlesec}
\usepackage[english]{babel}
\usepackage[utf8]{inputenc}
\usepackage[usenames,dvipsnames]{xcolor}  
\usepackage{lmodern,graphicx} 
\usepackage[toc,page]{appendix}
\usepackage{mathrsfs,amsthm,amssymb,amsfonts,amsmath}
\usepackage[all,poly]{xy} 
\xyoption{arrow}
\usepackage{hyperref}
\hypersetup{
	colorlinks=true,
	urlcolor=MidnightBlue, citecolor=MidnightBlue, linkcolor= MidnightBlue
}

\entrymodifiers={+!!<0pt,\fontdimen22\textfont2>} 

\usepackage[top=2.3cm,bottom=3cm,right=2.5cm,left=2.5cm]{geometry} 

\titleformat{\section}{\Large\sc}{\thesection.}{1em}{} 

\newtheoremstyle{plain}{\topsep}{\topsep}{\itshape}{}{\bf}{ $-$}{ }
{\thmname{#1}\thmnumber{ #2}\thmnote{ \normalfont(#3)}}

\newcommand{\theoremname}{Theorem} \newcommand{\defname}{Definition}
\newcommand{\lemname}{Lemma} \newcommand{\corname}{Corollary}
 \newcommand{\exname}{Example}  \newcommand{\rmqname}{Remark}

\newtheorem{theoreme}{\theoremname}[section]
\newtheorem*{theoreme*}{\theoremname}
\newtheorem{proposition}[theoreme]{Proposition}
\newtheorem{lemme}[theoreme]{\lemname}

\theoremstyle{definition}
\newtheorem{definition}[theoreme]{\defname}
\newtheorem{exemple}[theoreme]{\exname}
\newtheorem{rmq}[theoreme]{\rmqname}

\newcommand{\NN}{\mathbb N}

\newcommand{\RR}{\mathbb R}
\newcommand{\CC}{\mathbb C}

\newcommand{\cA}{\mathcal{A}} \newcommand{\cB}{\mathcal{B}}
\newcommand{\cC}{\mathcal{C}} \newcommand{\cD}{\mathcal {D}}
 \newcommand{\cF}{\mathcal{F}}
\newcommand{\cG}{\mathcal{G}} 
\newcommand{\cI}{\mathcal{I}} 
 
\newcommand{\cM}{\mathcal{M}} 
\newcommand{\cO}{\mathcal{O}} 
 \newcommand{\cR}{\mathcal{R}}

\newcommand{\cY}{\mathcal{Y}} 

\DeclareMathOperator{\ob}{ob}
\DeclareMathOperator{\tr}{tr}
\DeclareMathOperator*{\cotens}{\Box}
\DeclareMathOperator{\im}{Im}
\DeclareMathOperator{\vect}{span}

\newcommand{\ie}{\textit{i.e. }}
\newcommand{\tq}{\;\, ; \;}
\newcommand{\appl}[5][]{\begin{array}[t]{crcl}
{#1} &\!\! {#2} & \!\!\rightarrow\!\! & {#3} \!\!\\ 
&\!\! {#4} & \!\!\mapsto\!\! & {#5} \!\!
\end{array}}

\newcommand{\cnk}[2]{ \begin{scriptsize}\left( \begin{array}{c}
 \!\!\!\! {#1} \!\!\!\! \\ 
 \!\!\!\! {#2} \!\!\!\!
\end{array}  \right)\end{scriptsize}}

\AtBeginDocument{}
\AtBeginDocument{}

\title{Classification of bicovariant differential calculi over free orthogonal
Hopf algebras}
\date{}
\author{Manon Thibault de Chanvalon}

\begin{document}
\maketitle
\begin{center}
\textit{{Laboratoire de~Mathématiques (UMR 6620)}, Université Blaise Pascal, \\ Complexe universitaire des Cézeaux, 63171 Aubière Cedex, France.}
\end{center}
\begin{center}
\href{mailto:manon.thibault@math.univ-bpclermont.fr}{\tt manon.thibault@math.univ-bpclermont.fr}
\vspace{1em}
\end{center}

\begin{abstract}
We show that if two Hopf algebras are monoidally equivalent, then their categories of bicovariant differential
calculi are equivalent. We then classify,
for $q \in \CC^*$ not a root of unity, the finite dimensional bicovariant differential calculi over the Hopf algebra
$ \cO_q(SL_2)$.
Using a monoidal equivalence between free orthogonal Hopf algebras and $ \cO_q(SL_2)$
for a given $q$, this leads us to the classification of finite dimensional bicovariant differential calculi over  free orthogonal Hopf algebras.
\end{abstract}

\section*{Introduction}

The notion of \emph{differential calculus} over a Hopf algebra has been introduced by Woronowicz in~\cite{wor_cd}, with the purpose of giving a natural adaptation of
 differential geometry over groups, in the context of quantum groups.
An important question in this topic, is the classification of bicovariant differential calculi
over a given Hopf algebra, see for example~\cite{BS_class}, \cite{maj_class} or \cite{hs_cd_SLq}.

The aim of the present paper is to classify the finite dimensional
(first order) bicovariant differential calculi
over an important class of Hopf algebras, namely the \emph{free orthogonal Hopf algebras},
also called \emph{Hopf algebras associated
to non-degenerate bilinear forms}~\cite{DV_BE}.
Given an invertible matrix $E \in GL_n(\CC)$ with $n \geqslant 2$,
the free orthogonal Hopf algebra $\cB(E)$ associated with $E$ 
is  the universal Hopf algebra generated 
by a family of elements $(a_{ij})_{1\leqslant i,j \leqslant  n}$
submitted to the relations:
\[E^{-1} a^t E a =I_n  =
	a E^{-1}a^t E ,
\]
where $ a $ is the matrix $ (a_{ij})_{1\leqslant i,j \leqslant  n}$. Its coproduct, counit and antipode are defined by:
\[\Delta (a_{ij})  = \sum\limits_{k=1}^n a_{ik} \otimes a_{kj}, \qquad
\varepsilon (a)  = I_n, \qquad
S(a)  = E^{-1}a^tE.\]
The Hopf algebra $\cB(E)$ can also be obtained as an appropriate quotient of the FRT bialgebra associated
to Yang-Baxter operators constructed by Gurevich~\cite{gure}.

If $E \overline{E} = \lambda I_n$, with $\lambda \in \RR^*$,
 there exists an involution $*$ on $\cB(E)$
defined by $a_{ij}^* = \left(E^{-1} a^t E\right)_{ji}$, endowing
$\cB(E)$ with a Hopf $*$-algebra structure. This Hopf $*$-algebra corresponds
to a free orthogonal compact quantum group as defined in~\cite{wang_vandaele} or \cite{qortho}, 
and is generally denoted by $\cA_o((E^t)^{-1})$.
This justifies the term ``free orthogonal Hopf algebra'' for $\cB(E)$.

The starting point of our classification is a result of~\cite{bic_BE},
which states that if $q \in \CC^*$ satisfies
$q^2 + \tr(E^{-1}E^t)q + 1 = 0$, then the Hopf algebras
$\cB(E)$ and $\cO_q(SL_2)$ are monoidally equivalent, \ie their categories of
comodules are monoidally equivalent. The proof of~\cite{bic_BE},
is based on a deep result of Schauenburg~\cite{Sch_big},
and gives an explicit description of the correspondence between
$\cB(E)$-comodules and $\cO_q(SL_2)$-comodules.
We use here similar arguments to show that if two Hopf algebras
are monoidally equivalent, then their categories of bicovariant differential calculi are equivalent
(Theorem~\ref{th_eq}). This theorem generalizes a result of~\cite{MO}, where the two monoidally equivalent Hopf
algebras are assumed to be related by a cocycle twist. 
Applying Theorem~\ref{th_eq} to the Hopf algebras
$\cB(E)$ and $\cO_q(SL_2)$, the study of bicovariant differential calculi over the Hopf algebra  $\cB(E)$ is simplified, and
 therefore reduces to the study of bicovariant differential calculi
over $\cO_q(SL_2)$.

This classification 
 has been made over $\cO_q(SL_2)$ in~\cite{hs_cd_SLq}
for transcendental values of $q$
(which is not the case here since $q$ has to satisfy $q^2 + \tr(E^{-1}E^t)q + 1 = 0$).
Our classification uses a different approach than in~\cite{hs_cd_SLq}, and is based on the classification of the finite dimensional $ \cO_q(SL_2)$-Yetter-Drinfeld modules made in~\cite{TAK}. 

The paper is organized as follows.
We gather in the first section some known results about bicovariant differential calculi
over Hopf algebras, and their formulation in terms of Yetter-Drinfeld modules.
Furthermore, we show that if the category of Yetter-Drinfeld modules
over a Hopf algebra $H$ is semisimple, then the bicovariant differential calculi
over $H$ are inner.
In Section 2, using the language of cogroupoids~\cite{HG_co}, 
we prove that two  monoidally equivalent Hopf algebras
have equivalent categories of bicovariant differential calculi.
We finally classify in Section 3 the finite dimensional bicovariant differential calculi
over the Hopf algebra $\cO_q(SL_2)$ for $q \in \CC^*$ not a root of unity, using
the fact that by~\cite{TAK} the category of finite dimensional
Yetter-Drinfeld modules over $\cO_q(SL_2)$ is semisimple.
 This allows to classify the finite dimensional bicovariant differential calculi
over $\cB(E)$, provided that the solutions of the equation
$q^2 + \tr(E^{-1}E^t)q + 1 = 0$ are not roots of unity.

\subsection*{Notations and Conventions}
Let $H$ be a Hopf algebra.
Its comultiplication, antipode and counit will respectively be denoted by
$\Delta$, $S$ and $\varepsilon$.
A coaction of a left (respectively right) $H$-comodule will generally be denoted by $\lambda$
(respectively $\rho$). 

We will use Sweedler's notations:
$\Delta(x) = \sum x_{(1)} \otimes x_{(2)}$ for $x \in H$, 
and
$\rho(v) = \sum v_{(0)} \otimes v_{(1)}$ for $v$ in a right comodule $V$.

\section{Bicovariant differential calculi}

We start this section by recalling the definition of a bicovariant differential calculus,
and of the equivalent notion, expressed in terms of Yetter-Drinfeld modules (called \emph{reduced differential calculus}
in this paper).
We then prove some basic lemmas which will be useful in the sequel.
The main result of this section states that if the category of (finite dimensional) Yetter-Drinfeld modules
over a Hopf algebra $H$ is semisimple, then the  (finite dimensional)  bicovariant differential calculi
over $H$ are inner.

We refer to~\cite{KS} for background material on Hopf algebras and comodules.

\begin{definition}
Let $H$ be a Hopf algebra. A \emph{Hopf bimodule} $M$ over $H$
is an $H$-bimodule together with a
left comodule structure $\lambda: M \rightarrow H \otimes M$ and a right comodule structure
$\rho: M \rightarrow M \otimes H$ such that:
\begin{itemize}
\item $\forall x,y \in H, \forall v \in M$, $\lambda(x.v.y) = \Delta(x).\lambda(v).\Delta(y)$,
\item $\forall x,y \in H, \forall v \in M$, $\rho(x.v.y) = \Delta(x).\rho(v).\Delta(y)$,
\item $(id_H \otimes \rho) \circ \lambda = (\lambda \otimes id_H) \circ \rho$.
\end{itemize}
The category of Hopf bimodules over $H$, whose morphisms
are the maps which are right and left linear and colinear over $H$, is denoted
by $^H_H \cM^H_H$.
\end{definition}

\begin{definition}
Let $H$ be a Hopf algebra. A (right) \emph{Yetter-Drinfeld module} over $H$ is a right $H$-module and a right $H$-comodule $V$ such that:
\[\forall x \in H, \forall v \in V, \quad \sum (v.x)_{(0)} \otimes (v.x)_{(1)} = \sum v_{(0)}.x_{(2)} \otimes S(x_{(1)})v_{(1)}x_{(3)} .\]
The category of Yetter-Drinfeld modules over $H$, whose morphisms
are the maps which are both linear and colinear over $H$, is denoted
by $\cY\cD(H)$. The category  of finite dimensional Yetter-Drinfeld modules
over $H$ is denoted by $\cY\cD_f(H)$.
\end{definition}

\begin{exemple}\label{ex_eps}
Let $H$ be a Hopf algebra. We denote by $\CC_\varepsilon$ the Yetter-Drinfeld module
whose base-space is $\CC$, with right coaction $\lambda \mapsto \lambda \otimes 1$
and right module structure defined by
$\lambda \triangleleft x = \lambda \varepsilon(x)$ for $\lambda \in \CC_\varepsilon$ and
$x \in H$.
\end{exemple}

We recall from~\cite{YDmod} the correspondence between Yetter-Drinfeld modules and Hopf bimodules.

\begin{theoreme}[{\cite[Theorem 5.7]{YDmod}}]\label{th_YD}
Let $H $ be a Hopf algebra. The categories $^H _H \cM ^H _H$ and  $\cY \cD(H)$
are equivalent.
\end{theoreme}
We describe for convenience the equivalence of categories involved in the previous theorem.

Let $M$ be a Hopf bimodule over $H$, with right coaction $\rho$ and left coaction $\lambda$. The space $^{inv}M = \{v \in M \tq \lambda(v) = 1 \otimes v\}$
of left-coinvariant elements of $M$ has a Yetter-Drinfeld module structure defined
as follows. We have $\rho(^{inv}M) \subset \:^{inv}M \otimes H$,
and the right coaction of $^{inv}M$ is just the restriction of $\rho$ to  $^{inv}M$.
The right module structure is defined by $w\triangleleft x =\sum S(x_{(1)}).w.x_{(2)}$.

Conversely, given a Yetter-Drinfeld module $V$, then the space $H \otimes V$ can be equipped
with a Hopf bimodule structure, with left and right actions given by:
\[x.(y \otimes v).z = \sum xyz_{(1)} \otimes v\triangleleft z_{(2)},\]
and the right ($\rho$) and left ($\lambda$) coactions given by:
\begin{align*}
\rho(x \otimes v) 
	& = \sum x_{(1)} \otimes v_{(0)} \otimes x_{(2)}v_{(1)},\\
\lambda(x \otimes v) 
	& = \sum x_{(1)} \otimes x_{(2)} \otimes v.
\end{align*}
We then have for $M \in\, ^H_H\cM ^H _H$, $M \cong H \otimes\,^{inv}M$
and for $V \in \cY\cD(H)$, $V \cong \,^{inv}(H \otimes V)$.
The equivalence of categories between $^H _H \cM ^H _H$ and
$\cY\cD(H)$ is then:
\[
\appl[\cF:]{^H _H \cM ^H _H}{\cY \cD(H)}{M}{^{inv}M,} \quad \text{with quasi-inverse }
\appl[\cG:]{\cY \cD(H)}{^H _H \cM ^H _H}{V}{H \otimes V.}
\]
A morphism $f: M \rightarrow N$ in $^H _H \cM ^H _H$ automatically satisfies
$f(^{inv}M) \subset \,^{inv}N$, and $\cF(f) : \,^{inv}M \rightarrow \,^{inv}N$ is just the restriction of $f$. Conversely, if $f : V \rightarrow W$ is a morphism
of Yetter-Drinfeld modules, then $\cG(f) = id_H \otimes f$.

\begin{definition}
Let $H$ be a Hopf algebra. A (first order) \emph{bicovariant differential calculus} $(M,d)$ over $H$ is a Hopf bimodule $M$
together with a left and right comodule morphism $d: H \rightarrow M$ such that $\forall x,y \in H, d(xy) = x.d(y)+d(x).y$ and
such that $M = \vect \{x.d(y) \tq x,y \in H \}.$

A bicovariant differential calculus $(M,d)$ is said \emph{inner}
if there exists a bi-coinvariant element $\theta \in M$
(\ie satisfying $\rho(\theta) = \theta \otimes 1$
and $\lambda(\theta) = 1 \otimes \theta)$ such that
$\forall x \in H, d(x) = \theta.x - x. \theta$.

The \emph{dimension} of a bicovariant differential calculus  $(M,d)$ is the dimension of the vector
space $^{inv}M$.

A \emph{morphism of bicovariant differential calculi} $f : (M,d_M) \rightarrow (N,d_N)$ is a morphism of Hopf bimodules
such that $f \circ d_M = d_N$.

We denote by $\cD\cC(H)$ the category of bicovariant differential calculi over $H$.
\end{definition}

Bicovariant differential calculi were introduced by Woronowicz in~\cite{wor_cd}. An overview is given in~\cite[Part IV.]{KS}.
The notion of bicovariant differential calculus has the following interpretation in terms of Yetter-Drinfeld modules.

\begin{definition}
Let $H$ be a Hopf algebra. A \emph{reduced differential calculus} over $H$
is a Yetter-Drinfeld module $V$ together with a surjective map $\omega:H \rightarrow V$ satisfying:
\[\forall x,y \in H, \:\omega(xy) = \omega(x).y + \varepsilon(x) \omega(y) \quad \text{and} \quad
\sum \omega(x)_{(0)} \otimes \omega(x)_{(1)}  = \sum \omega(x_{(2)}) \otimes S(x_{(1)})x_{(3)}.\]

A \emph{morphism of reduced differential calculi} $f: (V, \omega_V) \rightarrow (W,\omega_W)$
is a morphism of Yetter-Drinfeld modules such that $f \circ \omega_V = \omega_W$.

We denote by $\cR\cD\cC(H)$ the category of reduced differential calculi over $H$.
\end{definition}

\begin{lemme}\label{lem_rdc}
The equivalence of categories of Theorem~\ref{th_YD} induces an equivalence between
the categories $\cD\cC(H)$ and $\cR\cD\cC(H)$:
\[
\appl[\cF:]{\cD\cC(H)}{\cR\cD\cC(H)}{(M,d)}{(^{inv}M,\omega_d)} \quad
\text{with quasi-inverse }
\appl[\cG:]{\cR\cD\cC(H)}{\cD\cC(H)}{(V,\omega)}{(H \otimes V, d_\omega)}
\]
where for $x \in H$, $\omega_d(x) = \sum S(x_{(1)})d(x_{(2)})$
and $d_\omega(x) = \sum x_{(1)} \otimes \omega(x_{(2)})$.
\end{lemme}

\begin{proof}
The one-to-one correspondence between bicovariant differential calculi and
reduced differential calculi is described in~\cite[Section 14]{KS}.
We may now focus on the functoriality of this correspondence.

If $f : (M,d_M) \rightarrow (N,d_N)$ is a morphism of bicovariant differential calculi, 
then the restriction of $f$, $\cF(f) :\,^{inv}M \rightarrow \,^{inv}N$
satisfies for all $x \in H$,
\[
\cF(f) \circ \omega_{d_M}(x) = \sum f\left(S(x_{(1)})d_M(x_{(2)})\right)
	= \sum S(x_{(1)})f(d_M(x_{(2)})) = \sum S(x_{(1)}) d_N(x_{(2)})
	= \omega_{d_N}(x).
\]
Hence $\cF(f)$ is a morphism of reduced differential calculi.

Conversely, if $f: (V, \omega_V) \rightarrow (W, \omega_W)$ is a morphism
of reduced differential calculi, then
\[(id_H \otimes f) \circ d_{\omega_V} (x) = \sum x_{(1)} \otimes f( \omega_V(x_{(2)}))
	=  \sum x_{(1)} \otimes \omega_W(x_{(2)}) = d_{\omega_W}(x).
\]
Thus $\cG(f) = id_H \otimes f$ is a morphism of bicovariant differential calculi.
 
Since $\cF$ and $\cG$ are quasi-inverse to each other between the categories $\cY\cD(H)$
and $^H _H \cM ^H _H$, it only remains to check 
that the natural transformation providing the equivalence
$\cF \circ \cG \cong id$ (respectively $\cG \circ \cF \cong id$) consists of morphisms
of reduced (respectively bicovariant) differential calculi.
Let $(V, \omega)$ be a reduced differential calculus over $H$.
The isomorphism of Yetter-Drinfeld modules
\[
	\appl[\theta :]{\cF \circ \cG(V) =\, ^{inv}(H \otimes V)}{V}{x \otimes v}{\varepsilon(x)v} 
\]
satisfies for $x \in H$,
\[
\theta \circ \omega_{d_\omega} (x) = \theta\left(\sum S(x_{(1)}) d_{\omega} (x_{(2)})\right)
	= \theta \left(\sum  S(x_{(1)}) x_{(2)} \otimes \omega (x_{(3)})\right)
	= \theta(1 \otimes \omega(x)) = \omega(x).
\]
Thus $\theta$ is an isomorphism of reduced differential calculi.

 Conversely, let $(M,d)$ be a bicovariant differential calculus over $H$. The isomorphism
of Hopf bimodules
\[
	\appl[\gamma :]{\cG \circ \cF (M) = H \otimes \,^{inv}M}{M}{x \otimes  v}{x.v}
\]
satisfies for $x \in H$,
\[
	\gamma \circ d_{\omega_d}(x)
	= \gamma \left( \sum x_{(1)} \otimes \omega_d(x_{(2)}) \right)
	=  \gamma \left( \sum x_{(1)} \otimes S(x_{(2)}) d (x_{(3)}) \right)
	= \sum \varepsilon(x_{(1)}) d(x_{(2)}) = d(x).
\]
Hence $\gamma$ is a morphism of bicovariant differential calculi, which ends the proof.
\end{proof}

\begin{rmq}\label{rmq_im}
Let $V$ be a Yetter-Drinfeld module, and let $\omega : H \rightarrow V$ be a map
satisfying all the axioms of a reduced differential calculus, except the surjectivity condition.
Then $\im(\omega)$ is a Yetter-Drinfeld submodule of $V$. Indeed, we have
$\omega(x).y = \omega(xy)- \varepsilon(x)\omega(y) =
\omega(xy-\varepsilon(x)y) \in \im(\omega) $ for all $x, y \in H$, thus
$\im(\omega)$ is a submodule of $V$, and
$\sum \omega(x)_{(0)} \otimes \omega(x)_{(1)}  = \sum \omega(x_{(2)}) \otimes S(x_{(1)})x_{(3)} \in \im(\omega) \otimes H$, thus
$\im(\omega)$ is a subcomodule of $V$.
\end{rmq}

\begin{definition}
A reduced differential calculus $(V, \omega)$ is said \emph{inner}
if there exists a coinvariant element $\theta \in V$
(\ie satisfying $\rho(\theta) = \theta \otimes 1$) such that
$\forall x \in H,$ $\omega(x) = \theta.x - \varepsilon(x) \theta$.

A reduced differential calculus $\omega : H \rightarrow V$ is said \emph{simple}
if $V$ is a simple Yetter-Drinfeld module. That is to say, if there is no non-trivial subspace
$W \subset V$, which is both a submodule and a subcomodule of $V$.

Let $(V,\omega)$, $(W_1, \omega_1)$,
($W_2, \omega_2)$ be reduced differential calculi. We say that
$(V, \omega)$ is the direct sum of  $(W_1, \omega_1)$ and
$(W_2, \omega_2)$ and we write $(V,\omega) = ( W_1, \omega_1) \oplus
(W_2, \omega_2)$, if $V = W_1 \oplus W_2$ and if
for all $x \in H$, $\omega(x ) = (\omega_1(x),\omega_2(x))$.
\end{definition}

Note that the direct sum of reduced differential calculi is not always well defined.
The problem is that if $(V, \omega_V)$ and $(W, \omega_W)$ are reduced differential calculi 
over a Hopf algebra $H$, then the map 
\[
	\appl[\omega:]{ H}{V \oplus W}{x}{(\omega_V(x), \omega_W(x))}
\]
can fail to be surjective. We give in the next lemma a necessary and a sufficient condition for the existence of the direct sum
of simple reduced differential calculi.

\begin{lemme}\label{lem_sum}
Let $(V_1, \omega_1), \ldots , (V_n, \omega_n)$ be simple reduced differential calculi over a Hopf algebra $H$. We set
\[
	\appl[\omega:]{H}{V =\bigoplus\limits_{i =1}^n V_i}{x}{(w_1(x), \ldots , w_n(x))}.
\]

If the $V_i$'s are two-by-two non isomorphic as Yetter-Drinfeld modules,
then $(V,\omega)$ is a reduced differential calculus.

Conversely, if $(V,\omega)$ is a reduced differential calculus, then the reduced differential calculi
 $(V_i, \omega_i)$
are two-by-two non-isomorphic.
\end{lemme}

\begin{proof}
The map $\omega$ clearly satisfies all the axioms of a reduced differential calculus, except the surjectivity condition.
In order to prove the lemma, we thus have to examine under which conditions $\omega$ is onto.
Assume that the $V_i$'s are two-by-two non-isomorphic (as Yetter-Drinfeld modules).
According to Remark~\ref{rmq_im}, the image of $\omega$ is a Yetter-Drinfeld submodule of $V$. 
There is therefore a subset
$\cI \subset \{1, \ldots, n\}$ such that there exists an isomorphism of Yetter-Drinfeld modules
$f : \im(\omega) \rightarrow \bigoplus\limits_{i\in\cI} V_i$.
For $k \in \{1, \ldots, n\}$, we denote by $\pi_k :V \rightarrow V_k$ the canonical projection. 
The map $\pi_k \circ \omega = \omega_k$ is onto, thus the restriction of
$\pi_k$ to $\im(\omega)$ is also onto.
This means that $\pi_k$ induces a non-zero morphism of Yetter-Drinfeld modules
$\bigoplus\limits_{i\in\cI} V_i \rightarrow V_k$, hence an isomorphism
of Yetter-Drinfeld modules $V_l \cong V_k$, with $l \in \cI$.
Since by hypothesis the $V_i$'s are two-by-two non-isomorphic,
we have $k= l$, hence $k \in \cI$. Thus $\cI = \{1, \ldots, n\}$, $\im(\omega) = V$, and we 
conclude that $\omega: H \rightarrow V$ is a reduced differential calculus.

Assume now that there is an isomorphism $f : (V_j,\omega_j) \rightarrow (V_i, \omega_i)$
with $i \neq j$. 
We denote by $\eta : H \rightarrow V_i \oplus V_i$ the map defined by the composition
\[\xymatrixcolsep{4em}\xymatrix{
	 H \ar[r]^-\omega & V \ar[r]^-{\pi_i \oplus \pi_j} & V_i \oplus V_j \ar[r]^-{id \oplus f} & V_i \oplus V_i .
}\] 
The map $\eta$ is clearly not surjective, since $\eta (x) = (\omega_i(x) , f\circ \omega_j(x)) = (\omega_i(x), \omega_i(x))$.
This implies that $\omega$ is not surjective,
since $\pi_i \oplus \pi_j$ and $id \oplus f$ are both surjective.
\end{proof}

\begin{lemme}\label{lem_simple}
Let $V$ be a simple Yetter-Drinfeld module over a Hopf algebra $H$, admitting a
non-zero right-coinvariant element $\theta \in V$.
If $V$ is not isomorphic to the Yetter-Drinfeld module $\CC_\varepsilon$
(of Example~\ref{ex_eps}),
then the map
\[
	\appl[\omega_\theta:]{H}{V}{x}{\theta.x-\varepsilon(x)\theta}
\]
defines a reduced differential calculus over $H$.
\end{lemme}

\begin{proof}
We have for $x \in H$,
\begin{align*}
\omega_\theta (x).y + \varepsilon(x)\omega_\theta(y)
	&= (\theta.x - \varepsilon(x)\theta).y + \varepsilon(x)(\theta.y-\varepsilon(y)\theta) \\
	& = (\theta.x).y - \varepsilon(x)\varepsilon(y)\theta = \omega_\theta (xy)
\end{align*}
and 
\begin{align*}
\rho \circ  \omega_\theta (x)& = \rho(\theta.x) -\varepsilon(x) \theta \otimes 1 \\
	&= \sum \theta.x_{(2)} \otimes S(x_{(1)}).1.x_{(3)}  -\varepsilon(x) \theta \otimes 1 \\
	& =  \sum \theta.x_{(2)} \otimes S(x_{(1)})x_{(3)} 
		- \sum \varepsilon(x_{(2)})\theta \otimes S(x_{(1)})x_{(3)} \\
	& = \sum \omega_\theta( x_{(2)}) \otimes S(x_{(1)})x_{(3)}.
\end{align*}
By Remark~\ref{rmq_im}, the image of $\omega_\theta$ is thus a
Yetter-Drinfeld submodule of $V$. Since $V$ is simple, the image of $\omega_\theta$ is either $V$, in which case
$\omega_\theta$ is indeed a reduced differential calculus, or $\im(\omega_\theta) = (0)$.
In that case, since $\theta$ is coinvariant and $\theta . x = \varepsilon(x) \theta$ for all $x \in H$, the map
$\mu : \CC_\varepsilon  \rightarrow V$ given by $\mu(\lambda) = \lambda \theta$
is a non-zero morphism between simple Yetter-Drinfeld modules, hence an isomorphism.
\end{proof}

The end of this section is devoted to the proof of the following lemma.

\begin{lemme}\label{lem_inner}
Let $H$ be a Hopf algebra such that the category $\cY\cD_f(H)$ is semisimple
(\ie each finite dimensional Yetter-Drinfeld module over $H$
can be decomposed into a direct sum of simple Yetter-Drinfeld modules).
Then each finite dimensional reduced differential calculus over $H$ is inner.
\end{lemme}

\begin{definition}
Let $(V,\omega)$ be a reduced differential calculus.
We denote by $V_\omega$ the Yetter-Drinfeld module over $H$ defined as follows.
As a right comodule, $V_\omega = V \oplus \CC$ (where the $H$-comodule structure
on $\CC$ is the canonical one: $\lambda \to \lambda \otimes 1$).
Its right module structure is defined for $v \in V$, $\lambda \in \CC$ and $x \in H$ by:
$(v, \lambda) .x = (v.x + \lambda \omega(x), \lambda \varepsilon(x))$.
Let us check that this formula defines an $H$-module structure on $V \oplus \CC$.
We have
\begin{align*}
((v, \lambda). x).y &=  (v.x + \lambda \omega(x), \lambda \varepsilon(x)).y
	=  ((v.x + \lambda \omega(x)).y + \lambda \varepsilon(x) \omega(y),
\lambda \varepsilon(x)\varepsilon(y)) \\
	& = (v.(xy) + \lambda \omega(xy), \lambda \varepsilon(xy)) = (v,\lambda).(xy),
\end{align*}
and the other axioms of a right module are clearly satisfied.
Before checking that the Yetter-Drinfeld condition is satisfied on $V_\omega$,
let us note that, denoting by $j : V \rightarrow V \oplus \CC$ the canonical
injection, and by $p: V \oplus \CC \rightarrow \CC_\varepsilon$ the canonical
projection, then clearly $j$ and $p$ are both module and comodule
maps, and the short sequence:
\[{
\xymatrix   {
	0 \ar[r]& V \ar[r]^-j & V_\omega \ar[r]^-p &  \CC_\varepsilon \ar[r] &0 
}}
\]
is exact.
Since $V$ is a Yetter-Drinfeld module and $j : V \rightarrow V_\omega$ is a module and comodule morphism, 
the Yetter-Drinfeld condition:
\[\forall x \in H, \rho(w.x) = \sum w _{(0)}.x _{(2)} \otimes S(x _{(1)}) w _{(1)} x _{(3)}\]
is automatically satisfied for $w \in j(V)$.
Hence it only remains to check that
the Yetter-Drinfeld condition is also satisfied on $\CC$, that is, that for all $x$ in $H$,
$\rho((0, 1).x) = \sum (0, 1).x_{(2)} \otimes S(x_{(1)})x_{(3)}$.
We have for $x \in H$
\begin{align*}
	\sum (0, 1).x_{(2)} \otimes S(x_{(1)})x_{(3)} 
		&= \sum (\omega(x_{(2)}), \varepsilon(x_{(2)}))  \otimes S(x_{(1)})x_{(3)} \\
		& =  (j\otimes id) \left(\sum \omega(x_{(2)})  \otimes S(x_{(1)})x_{(3)}\right)  + (0,1) \otimes \varepsilon(x) \\
		& = (j \otimes id) \circ  \rho (\omega(x)) + (0,1) \otimes \varepsilon (x) \\	
	&=  \rho(j(\omega(x))) +  \rho(0, \varepsilon(x))
	= \rho( \omega(x),  \varepsilon(x))  = \rho((0,1).x),
\end{align*}
hence $V_\omega$ is indeed a Yetter-Drinfeld module, 
and
\[{
\xymatrix   {
	0 \ar[r]& V \ar[r]^-j & V_\omega \ar[r]^-p &  \CC_\varepsilon \ar[r] &0 
}}
\]
is a short exact sequence of Yetter-Drinfeld modules.
\end{definition}

\begin{lemme}
A  reduced differential calculus
$(V,\omega)$ is inner  if and only if the short exact sequence
of Yetter-Drinfeld modules
\[{
\xymatrix   {
	0 \ar[r]& V \ar[r]^-j & V_\omega \ar[r]^-p &  \CC_\varepsilon \ar[r] &0 
}}
\]
splits.
\end{lemme}

\begin{proof}
Assume first that $(V,\omega)$ is inner.
Let $\theta \in V$ be a right-coinvariant element such that $\omega = x \mapsto \theta.x - \varepsilon(x)\theta$.
We set
\[
 \appl[r:]{V_\omega}{V}{(v, \lambda)}{v + \lambda \theta.}
\]
It is a comodule morphism since for $v \in V$, $ \lambda \in \CC$,
\begin{align*}
 (r \otimes id) \circ \rho(v, \lambda)  & = (r \otimes id) \circ \rho \circ j (v)
	+(r \otimes id )((0,\lambda) \otimes 1)
	= ((r \circ j) \otimes id) \circ \rho (v) +\lambda \theta \otimes 1 \\
	& = \rho(v) + \lambda \rho(\theta) =\rho \circ r (v,\lambda).
\end{align*}
And we have
for $v \in V$, $\lambda \in \CC$ and $x \in H$,
\[
r((v, \lambda).x) = r(v.x +\lambda \omega(x), \lambda\varepsilon(x))
	= v.x + \lambda \theta.x -\lambda\varepsilon(x)\theta + \lambda\varepsilon(x)\theta
	= (v+\lambda \theta).x = r(v, \lambda).x.
\]
 Hence $r$ is a Yetter-Drinfeld module morphism satisfying $r\circ j = id_V$, so that the above sequence
splits.

Assume conversely that the short exact sequence of Yetter-Drinfeld modules
associated to $(V,\omega)$ splits:
\[{
\xymatrix   {
	0 \ar[r]& V \ar[r]^-j & V_\omega =V \oplus \CC \ar@/^1.1pc/[l]^-r\ar[r]^-p  &  \CC_\varepsilon  \ar[r] &0 .
}}
\]
We set $\theta = r(0,1)$.
Then $\rho(\theta) =\rho\circ r(0,1)
= (r\otimes id)\circ \rho (0,1) = \theta \otimes 1$ and for $x \in H$,
\[
	\theta.x -\varepsilon(x)\theta = r((0,1).x) - r(0, \varepsilon(x))
		= r\left((\omega(x), \varepsilon(x))-(0, \varepsilon(x))\right) = r(j(\omega(x))) = \omega(x).
\]
Hence the result.
\end{proof}

Lemma~\ref{lem_inner} follows immediately.

\section{Monoidal equivalence}

We show in this section that if two Hopf algebras are monoidally equivalent,
then their categories of bicovariant differential calculi are also equivalent.
In order to describe the equivalence between the categories $\cD\cC(H)$ and
$\cD\cC(L)$, when $H$ and $L$ are monoidally equivalent Hopf algebras, we
will need some definitions and results about cogroupoids, which we recall here.
We refer to~\cite{HG_co} for a survey on the subject. 
\begin{definition}
A \emph{cocategory} $\cC$ consists of:
\begin{itemize}
\item a set of objects $\ob(\cC)$,
\item for all $X,Y \in \ob(\cC)$, an algebra $\cC(X,Y)$,
\item for all $X,Y,Z \in \ob(\cC)$, algebra morphisms $\Delta_{X,Y}^Z : \cC(X,Y) \rightarrow \cC(X,Z) \otimes \cC(Z,Y)$
and $\varepsilon_X : \cC(X,X) \rightarrow \CC$ such that for all $X,Y,Z,T \in \ob(\cC)$, 
the following diagrams commute:
\[{
\xymatrix  @!0 @R=1.5cm @C6.5cm {
	\cC(X,Y) \ar[r]^-{\Delta_{X,Y}^Z} \ar[d]_{\Delta_{X,Y}^T} & \cC(X,Z) \otimes \cC(Z,Y) \ar[d]^{id \otimes \Delta_{Z,Y}^T}  \\
	\cC(X,T) \otimes \cC(T,Y) \ar[r]_-{\Delta_{X,T}^Z \otimes id} &  \cC(X,Z) \otimes \cC(Z,T) \otimes \cC(T,Y) 
}}\]  
\[
\xymatrix   {
	\cC(X,Y) \ar@{=}[rd]  \ar[r]^-{\Delta_{X,Y}^X} & \cC(X,X) \otimes \cC(X,Y) \ar[d]^{\varepsilon_X \otimes id}  \\
	 &  \cC(X,Y) 
} \qquad \quad
\xymatrix   {
	\cC(X,Y) \ar@{=}[rd]  \ar[r]^-{\Delta_{X,Y}^Y} & \cC(X,Y) \otimes \cC(Y,Y) \ar[d]^{id \otimes \varepsilon_Y}  \\
	 &  \cC(X,Y) 
}
\]  
\end{itemize}
A cocategory is said to be \emph{connected} if for all $X, Y \in \ob(\cC)$, $\cC(X,Y)$ is a
non-zero algebra.
\end{definition}

\begin{definition}
A \textit{cogroupoid} $\cC$ is a cocategory equipped with linear maps $S_{X,Y} : \cC(X,Y) \rightarrow \cC(Y,X)$ 
such that for all $X,Y \in \ob(\cC)$, the following diagrams commute:
\[{
\xymatrix  @!0 @R=1.5cm @C2.6cm {
	\cC(X,X) \ar[r]^-{\varepsilon_X} \ar[d]_{\Delta_{X,X}^Y} & \CC \ar[r]^-{u} &  \cC(Y,X)  \\
	\cC(X,Y) \otimes \cC(Y,X) \ar[rr]_-{S_{X,Y} \otimes id} & &  \cC(Y,X) \otimes \cC(Y,X)  \ar[u]^m
}}
\]
\[
{
\xymatrix  @!0 @R=1.5cm @C2.6cm {
	\cC(X,X) \ar[r]^-{\varepsilon_X} \ar[d]_{\Delta_{X,X}^Y} & \CC \ar[r]^-{u} &  \cC(X,Y)  \\
	\cC(X,Y) \otimes \cC(Y,X) \ar[rr]_-{id \otimes S_{Y,X}} & &  \cC(X,Y) \otimes \cC(X,Y)  \ar[u]^m
}}
\] 
where $m$ denotes the multiplication and $u$ the unit.
\end{definition}

We will use Sweedler notations for cogroupoids:
\[ \text{for } a^{X,Y} \in \cC(X,Y), \;\Delta_{X,Y}^Z(a^{X,Y}) = \sum a_{(1)}^{X,Z} \otimes a_{(2)}^{Z,Y}.\]

\begin{theoreme}[{\cite[Proposition 1.16 and Theorem 6.1]{HG_co}}]\label{bic_YD}
Let $H$ and $L$ be two Hopf algebras such that there exists a linear monoidal equivalence
between their categories of right comodules $\cM^H$ and $\cM^L$. Then there exists a linear monoidal equivalence between $\cY\cD(H)$ and $\cY\cD(L)$, inducing 
an equivalence between the categories of finite dimensional Yetter-Drinfeld modules $\cY\cD_f(H)$ and $\cY\cD_f(L).$
\end{theoreme}

Let us recall the construction of this equivalence.
As a consequence of~\cite{Sch_big}, restated in the context of cogroupoids, the existence of a  linear monoidal equivalence between the categories
$\cM^H$ and $\cM^L$ is equivalent to the existence of a connected cogroupoid $\cC$
and two objects $X, Y \in \ob (\cC)$
such that $H \cong \cC(X, X)$ and $L \cong \cC(Y,Y)$ (see~\cite[Theorem 2.10]{HG_co}).
Then the equivalence between the categories  $\cY\cD(H)$ and $\cY\cD(L)$
is given by the functor:
\[
	\appl[\cF_X^Y:]{\cY\cD(\cC(X,X))}{\cY\cD(\cC(Y,Y))}{V}{V  \cotens\limits_{\cC(X,X)} \cC(X,Y),}
\]
where 
\[V \! \cotens\limits_{\cC(X,X)}\! \cC(X,Y) 
= \left\{ \sum\limits_i v_i \otimes a_i^{X,Y}  \!\in V \otimes \cC(X,Y) \, ;
	\sum v_{i(0)} \otimes v_{i(1)}^{X,X} \otimes a_i^{X,Y}
	= \sum v_i \otimes a_{i(1)}^{X,X} \otimes a_{i(2)}^{X,Y} \right\}.\]
The right $L \cong \cC(Y,Y)$-module structure of  $V  \cotens\limits_{\cC(X,X)} \cC(X,Y) $
is given by:
\[
	\left(\sum_i v_i\otimes a_i^{X,Y}\right) \triangleleft b^{Y,Y}
	= \sum_i v_i.b_{(2)}^{X,X} \otimes S_{Y,X}(b_{(1)}^{Y,X})a_i b_{(3)}^{X,Y}
\]
and its right comodule structure is given by the map $id_V \otimes \Delta_{X,Y}^Y$.
The quasi-inverse of $\cF_X^Y$ is the functor $\cF_Y^X$.
By~\cite[Proposition 1.16]{HG_co}, the functor $\cF_X^Y$ induces
an equivalence between the categories of finite dimensional Yetter-Drinfeld modules $\cY\cD_f(H)$ and $\cY\cD_f(L).$

\begin{lemme}
Let $\cC$ be a cogroupoid and let $X,Y$ be in $\ob(\cC)$ such that $\cC(Y,X) \neq (0)$.
Let  $\omega : \cC(X,X) \rightarrow V$ be a reduced differential calculus over $\cC(X,X)$.
The map
\[
	 \appl[\overline{\omega}:]{\cC(Y,Y)}{V \cotens\limits_{\cC(X,X)} \cC(X,Y)}{a^{Y,Y}}
{ \sum \omega(a_{(2)}^{X,X}) \otimes S_{Y,X}(a_{(1)}^{Y,X})a_{(3)}^{X,Y}}
\]
is a reduced differential calculus over $\cC(Y,Y)$.
\end{lemme}

\begin{proof}
We already know, by the previous theorem, that $V \cotens\limits_{\cC(X,X)} \cC(X,Y)$
is a Yetter-Drinfeld module over $\cC(Y,Y)$. We firstly have to check 
that the map $\overline{\omega}$ is well defined, which is to say, we have to
check that
\[
\sum \omega(a_{(2)}^{X,X})_{(0)} \otimes \omega(a_{(2)}^{X,X})_{(1)} \otimes S_{Y,X}(a_{(1)}^{Y,X})a_{(3)}^{X,Y} 
	=  \sum \omega(a_{(2)}^{X,X}) \otimes \Delta_{X,Y}^X \left(S_{Y,X}(a_{(1)}^{Y,X})a_{(3)}^{X,Y}\right). \]
On the one hand, we have:
\begin{align*}
\sum \omega(a_{(2)}^{X,X})_{(0)} \otimes \omega(a_{(2)}^{X,X})_{(1)} \otimes S_{Y,X}(a_{(1)}^{Y,X})a_{(3)}^{X,Y}
	= \sum \omega(a_{(3)}^{X,X}) \otimes S_{X,X}(a_{(2)}^{X,X})a_{(4)}^{X,X} \otimes S_{Y,X}(a_{(1)}^{Y,X})a_{(5)}^{X,Y}.
\end{align*}
And on the other hand,
\begin{align*}
\Delta_{X,Y}^X \left(S_{Y,X}(a^{Y,X})b^{X,Y}\right) 
	&= \Delta_{X,Y}^X (S_{Y,X}(a^{Y,X}))\Delta_{X,Y}^X (b^{X,Y}) \\
	& =\sum \left(S_{X,X}(a_{(2)}^{X,X}) \otimes S_{Y,X}(a_{(1)}^{Y,X})\right)
		\left(b_{(1)}^{X,X} \otimes b_{(2)}^{X,Y} \right) \\
	& = \sum S_{X,X}(a_{(2)}^{X,X})b_{(1)}^{X,X} \otimes S_{Y,X}(a_{(1)}^{Y,X})b_{(2)}^{X,Y}
\end{align*}
so that 
\begin{align*}
\sum \omega(a_{(2)}^{X,X}) \otimes \Delta_{X,Y}^X \left(S_{Y,X}(a_{(1)}^{Y,X})a_{(3)}^{X,Y}\right)
	= \sum  \omega(a_{(3)}^{X,X}) \otimes
		S_{X,X}(a_{(2)}^{X,X})a_{(4)}^{X,X} \otimes S_{Y,X}(a_{(1)}^{Y,X})a_{(5)}^{X,Y}
\end{align*}
which shows that $\overline{\omega} : \cC(Y,Y) \rightarrow V \cotens\limits_{\cC(X,X)} \cC(X,Y)$
is well defined.

We have
\begin{align*}
\overline{\omega}(a^{Y,Y})  \triangleleft b^{Y,Y} 
 & =  \sum \left(\omega(a_{(2)}^{X,X}) \otimes S_{Y,X}(a_{(1)}^{Y,X})a_{(3)}^{X,Y}\right)
 \triangleleft b^{Y,Y} \\
	& = \sum \omega(a_{(2)}^{X,X}).b_{(2)}^{X,X} \otimes S_{Y,X}(b_{(1)}^{Y,X}) S_{Y,X}(a_{(1)}^{Y,X})a_{(3)}^{X,Y} b_{(3)}^{X,Y}.
\end{align*}
Consequently, we have
\begin{align*}
\overline{\omega}(a^{Y,Y}b^{Y,Y}) 
	& =  \sum \omega(a_{(2)}^{X,X}b_{(2)}^{X,X}) \otimes S_{Y,X}(a_{(1)}^{Y,X}b_{(1)}^{Y,X})a_{(3)}^{X,Y}b_{(3)}^{X,Y} \\
	& = \sum \omega(a_{(2)}^{X,X}).b_{(2)}^{X,X} \otimes S_{Y,X}(a_{(1)}^{Y,X}b_{(1)}^{Y,X})a_{(3)}^{X,Y}b_{(3)}^{X,Y} \\
		& \qquad +  \sum \varepsilon_X(a_{(2)}^{X,X}) \omega(b_{(2)}^{X,X}) \otimes S_{Y,X}(a_{(1)}^{Y,X}b_{(1)}^{Y,X})a_{(3)}^{X,Y}b_{(3)}^{X,Y} \\
	& = \overline{\omega}(a^{Y,Y})  \triangleleft b^{Y,Y} 
		+\sum  \omega(b_{(2)}^{X,X}) \otimes S_{Y,X}
(b_{(1)}^{Y,X})S_{Y,X}(a_{(1)}^{Y,X})a_{(2)}^{X,Y}b_{(3)}^{X,Y} \\
	& =  \overline{\omega}(a^{Y,Y})  \triangleleft b^{Y,Y} 
		+ \varepsilon_Y(a^{Y,Y}) \overline{\omega}(b^{Y,Y}).
\end{align*}
Denoting $\rho = id_V \otimes \Delta_{X,Y}^{Y}$ the $\cC(Y,Y)$-comodule
structure of $V \cotens\limits_{\cC(X,X)} \cC(X,Y)$, we have
for all $a^{Y,Y} \in \cC(Y,Y)$,
\begin{align*}
\rho \circ \overline{\omega}(a^{Y,Y}) 
	& =\sum \omega(a_{(2)}^{X,X}) \otimes \Delta_{X,Y}^Y \left(S_{Y,X}(a_{(1)}^{Y,X})a_{(3)}^{X,Y}\right) \\
	&= \sum  \omega(a_{(3)}^{X,X}) \otimes
	S_{Y,X}(a_{(2)}^{Y,X})a_{(4)}^{X,Y} \otimes S_{Y,Y}(a_{(1)}^{Y,Y})a_{(5)}^{Y,Y} \\
	& = \sum \overline{\omega}(a_{(2)}^{Y,Y}) \otimes S_{Y,Y}(a_{(1)}^{Y,Y})a_{(3)}^{Y,Y}.
\end{align*}

Now, in order to prove the lemma, it only remains to check that $\overline{\omega}$ is onto.
Let $\sum\limits_{i} v_i \otimes a_i^{X,Y}$ be in $ V  \cotens\limits_{\cC(X,X)} \cC(X,Y)$ and let $\varphi : \cC(Y,X) \rightarrow \CC$
be a linear map satisfying $\varphi(1) = 1$. We have
\[
	\sum\limits_{i} v_{i(0)} \otimes v_{i(1)}^{X,X}  \otimes a_i^{X,Y}
		= \sum\limits_{i} v_{i} \otimes a_{i(1)}^{X,X}  \otimes a_{i(2)}^{X,Y}
\]
since $\sum\limits_{i} v_i \otimes a_i^{X,Y}$ is in $ V  \cotens\limits_{\cC(X,X)} \cC(X,Y)$.
Applying $id_V \otimes \Delta_{X,X}^Y \otimes id_{\cC(X,Y)}$ on both sides, we find
\[
	\sum\limits_{i} v_{i(0)} \otimes v_{i(1)}^{X,Y} \otimes v _{i(2)}^{Y,X} \otimes a_i^{X,Y}
		= \sum\limits_{i} v_{i} \otimes a_{i(1)}^{X,Y}  \otimes a_{i(2)}^{Y,X}  \otimes a_{i(3)}^{X,Y}.
\]
This shows that
\begin{align*}
\sum\limits_{i}\varphi \left( v _{i(2)}^{Y,X} S_{X,Y}( a_i^{X,Y})\right) v_{i(0)} \otimes v_{i(1)}^{X,Y}
	& =  \sum\limits_{i} \varphi \left(  a_{i(2)}^{Y,X}  S_{X,Y}( a_{i(3)}^{X,Y}) \right) v_{i} \otimes a_{i(1)}^{X,Y} \\
	& =   \sum\limits_{i}\varepsilon_{Y} ( a_{i(2)}^{Y,Y})  v_{i} \otimes a_{i(1)}^{X,Y}
		= \sum\limits_{i}  v_{i} \otimes a_{i}^{X,Y}.
\end{align*}
Since $\omega : \cC(X,X) \rightarrow V$ is onto, there exists $b_i^{X,X} \in \cC(X,X)$ such that
$\omega(b_i^{X,X}) = v_i$.

We have then
\[
\sum v_{i(0)} \otimes v_{i(1)}^{X,X}  = \sum \omega(b_i^{X,X})_{(0)} \otimes  \omega(b_i^{X,X})_{(1)} 
	 = \sum  \omega(b_{i(2)}^{X,X}) \otimes S_{X,X}(b_{i(1)}^{X,X})b_{i(3)}^{X,X},
\]
so that 
\begin{align*}
\sum v_{i(0)} \otimes v_{i(1)}^{X,Y} \otimes v_{i(2)}^{Y,X}
	& =\sum \omega(b_{i(3)}^{X,X}) \otimes
S_{Y,X}(b_{i(2)}^{Y,X})b_{i(4)}^{X,Y} \otimes S_{X,Y}(b_{i(1)}^{X,Y})b_{i(5)}^{Y,X} .
\end{align*}

We have therefore
\begin{align*}
\sum\limits_{i}  v_{i} \otimes a_{i}^{X,Y}
	& = \sum_{i}\varphi \left( v _{i(2)}^{Y,X} S_{X,Y}( a_i^{X,Y})\right) v_{i(0)} \otimes v_{i(1)}^{X,Y} \\
	& =  \sum_{i}\varphi \left(S_{X,Y}(b_{i(1)}^{X,Y})b_{i(5)}^{Y,X} S_{X,Y}( a_i^{X,Y})\right)
		 \omega(b_{i(3)}^{X,X}) \otimes S_{Y,X}(b_{i(2)}^{Y,X})b_{i(4)}^{X,Y} \\
	& = \sum_{i}\varphi \left(S_{X,Y}(b_{i(1)}^{X,Y})b_{i(3)}^{Y,X} S_{X,Y}( a_i^{X,Y})\right)
		\overline{ \omega}(b_{i(2)}^{Y,Y})
\end{align*}
which allows to conclude that $\overline{\omega}$ is onto.
\end{proof}

\begin{rmq}\label{rmq_inner}
If $\omega : \cC(X,X) \rightarrow V$ is an inner reduced differential calculus,
let $\theta \in V$ be a right-coinvariant element such that
$\forall a^{X,X} \in \cC(X,X)$, 
$\omega(a^{X,X}) = \theta.a^{X,X} - \varepsilon_X(a^{X,X})\theta$.
We then have
\begin{align*}
\forall a^{Y,Y} \in \cC(Y,Y), \;\overline{\omega}(a^{Y,Y})
	&= \sum \omega(a_{(2)}^{X,X}) \otimes S_{Y,X}(a_{(1)}^{Y,X})a_{(3)}^{X,Y} \\
	& = \sum  (\theta.a_{(2)}^{X,X} - \varepsilon_X(a_{(2)}^{X,X})\theta) \otimes S_{Y,X}(a_{(1)}^{Y,X})a_{(3)}^{X,Y} \\
	& =  \sum  \theta.a_{(2)}^{X,X} \otimes S_{Y,X}(a_{(1)}^{Y,X})a_{(3)}^{X,Y} 
		- \theta \otimes \varepsilon_Y(a^{Y,Y})\\
	& = (\theta \otimes 1) \triangleleft a^{Y,Y} - \varepsilon_Y(a^{Y,Y})(\theta \otimes 1).
\end{align*}
Consequently, $\overline{\omega}$ is an inner reduced differential calculus, whose corresponding right-coinvariant element is $\theta \otimes 1.$
\end{rmq}

Combining the previous lemma with Theorem~\ref{bic_YD}, we obtain the main result of this section. It generalizes a result of~\cite{MO}, where the two monoidally equivalent Hopf
algebras are assumed to be related by a cocycle twist. 

\begin{theoreme}\label{th_eq}
Let $H$ and $L$ be two Hopf algebras such that there exists a linear monoidal equivalence
between their categories of right comodules $\cM^H$ and $\cM^L$.
Then there exists an equivalence between the categories:
\begin{itemize}
\item of bicovariant differential calculi $\cD\cC(H)$ and $\cD\cC(L)$,
\item of finite dimensional bicovariant differential calculi $\cD\cC_f(H)$ and $\cD\cC_f(L)$.
\end{itemize}
\end{theoreme}

\begin{proof}
Let $\cC$ be a connected cogroupoid such that there exist $X,Y \in \ob(\cC)$
satisfying $\cC(X,X) \cong H$ and $\cC(Y,Y) \cong L$.
We consider the functor induced by Theorem~\ref{bic_YD} and the previous lemma:
\[
	\appl[\cF_X^Y:]{\cR\cD\cC(H)}{\cR\cD\cC(L)}{(V,\omega)}%
{(V \cotens\limits_{H} \cC(X,Y), \overline{\omega})}
\]
which sends a morphism $f:V \rightarrow W$ in $\cR\cD\cC(H)$, to
$\cF_X^Y(f) = f\otimes id : V \cotens\limits_{H} \cC(X,Y) \rightarrow W \cotens\limits_{H} \cC(X,Y)$.
It is known to be a morphism of Yetter-Drinfeld modules, and one easily checks
that it is a morphism of reduced differential calculi.

Since $\cF_X^Y$ is an equivalence between the categories of Yetter-Drinfeld modules
over $H$ and $L$, with quasi-inverse $\cF_Y^X$, we only have to check 
that the natural transformation providing the equivalence
$\cF_X^Y \circ \cF_Y^X \cong id$ consists of morphisms of reduced differential calculi.
In other words, we have to check that, for all $(V,\omega) \in \cR\cD\cC(H)$,
 the morphism of Yetter-Drinfeld modules:
\[
	\appl[\theta_V :]{V}{(V \cotens\limits_{H} \cC(X,Y)) \cotens\limits_{L} \cC(Y,X)}
{v}{\sum v_{(0)} \otimes v_{(1)}^{X,Y} \otimes v_{(2)}^{Y,X}}
\]
is a morphism of reduced differential calculi.
We have for $a^{X,X} \in H \cong \cC(X,X)$,
\begin{align*}
\theta_V \circ \omega(a^{X,X})
	& = \sum(id \otimes \Delta_{X,X}^Y)\left(\omega(a_{(2)}^{X,X}) \otimes 
		S_{X,X}(a_{(1)}^{X,X})a_{(3)}^{X,X}\right) \\
	& = \sum  \omega(a_{(3)}^{X,X}) \otimes
	S_{Y,X}(a_{(2)}^{Y,X})a_{(4)}^{X,Y} \otimes S_{X,Y}(a_{(1)}^{X,Y})a_{(5)}^{Y,X} \\
	& =  \sum  \overline{\omega}(a_{(2)}^{Y,Y}) \otimes S_{X,Y}(a_{(1)}^{X,Y})a_{(3)}^{Y,X} = \overline{\overline{\omega}}(a^{X,X}).
\end{align*}
Thus $\theta_V$ is a morphism of reduced differential calculi, and $\cF_X^Y$
is an equivalence of categories.
Gathering this with Lemma~\ref{lem_rdc}, we obtain an
equivalence $\cD\cC(H) \cong \cR\cD\cC(H) \cong  \cR\cD\cC(L) \cong  \cD\cC(L)$,
inducing an equivalence  $\cD\cC_f(H) \cong \cR\cD\cC_f(H) \cong  \cR\cD\cC_f(L) \cong  \cD\cC_f(L)$.
\end{proof}

\section{Classification of bicovariant differential calculi over free orthogonal Hopf algebras}

In this section, we gather the results of the previous sections in order to classify
the finite dimensional reduced differential calculi over the free orthogonal Hopf algebras.
To this end, we start by classifying the finite dimensional reduced differential calculi over the Hopf algebra $ \cO_q(SL_2)$, when $q \in \CC^*$ is not a root of unity.
This classification is based on the classification of finite dimensional $ \cO_q(SL_2)$-Yetter-Drinfeld modules made in~\cite{TAK}, and Lemma~\ref{lem_inner}..

\begin{definition}
Let $q \in \CC^*$ be not a root of unity.
$ \cO_q(SL_2)$ is the Hopf algebra generated by four elements $a,b,c,d$ 
subject to the relations:
\[
	\left\{ \begin{array}{l}
         	ba = q ab \; , \:ca = qac \; , \:db = qbd \; , \:dc = q cd \; ,  \:bc =  cb \; , \\
		ad -q^{-1}bc = da-qbc = 1.
         \end{array} \right.
\]
Its comultiplication, counit and antipode are defined by:
\[\begin{array}{llll}
\Delta (a) = a \otimes a + b  \otimes c, & \Delta (b)  = a \otimes b + b \otimes d, & \Delta (c) = c \otimes a + d \otimes c, &
\Delta (d) =  c \otimes b + d \otimes d, \\
\varepsilon (a) = \varepsilon(d) = 1, & \varepsilon(b) = \varepsilon(c) = 0 , &&\\
S(a) = d, & S(b ) = -qb,  & S(c) = -q^{-1}c, & S(d) = a.
\end{array}\]
\end{definition}

\begin{definition}
Let $n$ be in $\NN$. We denote by $V_n$ the simple right $ \cO_q(SL_2)$-comodule
with basis $(v_i^{(n)})_{0 \leqslant i \leqslant n}$, and coaction $\rho_n$ defined by:
\[
	\rho_n(v_i^{(n)}) = \sum_{k=0}^n v_k^{(n)} \otimes
		\left( \sum_{\substack{r+s=k \\ 0\leqslant r \leqslant i \\0 \leqslant s \leqslant n-i}} \cnk{i}{r}_{q^2}\cnk{n-i}{s}_{q^2}q^{(i-r)s}a^rb^sc^{i-r}d^{n-i-s}
	\right)
\]
where $\cnk{n}{k}_{q^2}$ denotes the $q^2$-binomial coefficient.
That is to say:
\[
\cnk{n}{k}_{q^2} = q^{k(n-k)}\frac{[n]_q!}{[n-k]_q![k]_q!}
\quad \text{with } [k]_q = \frac{q^k-q^{-k}}{q-q^{-1}}
\quad \text{and } [k]_q! = [1]_q.[2]_q \ldots [k]_q.
\]
\end{definition}

\begin{definition}
Let $n,m$ be in $\NN$ and let $\epsilon \in \{-1,1\}.$
We denote by $V_{n,m}^\epsilon$ the $ \cO_q(SL_2)$-Yetter-Drinfeld module
$V_n \otimes V_m$ equipped with its canonical right coaction, and with right module structure
defined by:
\begin{align*}
(v_i^{(n)} \otimes v_j^{(m)}).a 
	& =  \epsilon q^{\frac{m-n}{2}+i-j}v_i^{(n)} \otimes v_j^{(m)}, \\
(v_i^{(n)} \otimes v_j^{(m)}).b 
	&= -\epsilon q^{-\frac{n+m}{2}+i+j+1}(1-q^{-2})[j]_q v_i^{(n)} \otimes v_{j-1}^{(m)},\\
(v_i^{(n)} \otimes v_j^{(m)}).c
	&= \epsilon q^{\frac{m+n}{2}-i-j}(1-q^{-2})[n-i]_q v_{i+1}^{(n)} \otimes v_j^{(m)}, \\
(v_i^{(n)} \otimes v_j^{(m)}).d
	& = \epsilon q^{\frac{n-m}{2}+j-i} (v_i^{(n)} \otimes v_j^{(m)}
		-q(1-q^{-2})^2[j]_q[n-i]_q v_{i+1}^{(n)} \otimes v_{j-1}^{(m)}).
\end{align*}
$V_{n,n}^\epsilon$ will also be denoted by $V_n^\epsilon$.
\end{definition}

\begin{rmq}\label{rmq_class}
By~\cite{TAK}, every simple finite dimensional $ \cO_q(SL_2)$-Yetter-Drinfeld module is of the form $V_{n,m}^\epsilon$, and each finite dimensional  $ \cO_q(SL_2)$-Yetter-Drinfeld module
can be decomposed into a direct sum of simple Yetter-Drinfeld modules.
To see that our description of $V_{n,m}^\epsilon$ coincides with the one
given in~\cite[(6.4)]{TAK}, just consider the basis
$(v_{i,j})_{\substack{0\leqslant i \leqslant n \\ 0\leqslant j \leqslant m}}$
given by
\[v_{i,j} = \dfrac{1}{[n-i]_q![m-j]_q!} v_{n-i}^{(n)} \otimes v_j^{(m)}.
\]
One can check that the $v_{i,j}$'s satisfy~\cite[(6.4)]{TAK} 
and that the map
\[v_{i,j} \mapsto \dfrac{1}{[n-i]_q![m-j]_q!} v_{n-i}^{(n)} \otimes v_j^{(m)}\]
is an isomorphism of Yetter-Drinfeld modules.
\end{rmq}

\begin{rmq}~\label{rmq_simple}
Let $n,m$ be in $\NN$ and $\epsilon$ be in  $\{-1, 1\}$. 
The Clebsch-Gordan formula for the decomposition of $V_n \otimes V_m$ into simple comodules ensures that the space of right-coinvariant elements of $V_{n,m}^\epsilon$ is one-dimensional if $n= m$, and zero-dimensional otherwise. Hence
if $n\neq m$, there is no inner reduced differential calculus of the form $\omega : \cO_q(SL_2) \rightarrow V_{n,m}^\epsilon$,
and there is at most one (up to isomorphism) 
inner reduced differential calculus of the form $\omega : \cO_q(SL_2) \rightarrow V_n^\epsilon$.
If $(n,\epsilon) \neq (0,1)$,
then $V_n^\epsilon$ is not isomorphic to the Yetter-Drinfeld module $\CC_\varepsilon$, and by  Lemma~\ref{lem_simple},
there indeed exists such an inner reduced differential calculus, which we denote by
 $\omega_n^\epsilon : \cO_q(SL_2) \rightarrow V_n^\epsilon$.
\end{rmq}

As a direct consequence of Lemma~\ref{lem_inner}, and the fact that by~\cite{TAK},
the category $\cY\cD_f(\cO_q(SL_2))$ is semisimple, we have the following result.

\begin{proposition}
Each finite dimensional bicovariant differential calculus over $\cO_q(SL_2)$ is inner.
\end{proposition}

This allows to deduce the classification of finite dimensional reduced differential
calculi over $\cO_q(SL_2)$.

\begin{theoreme}\label{th_class}
Every simple finite dimensional reduced differential calculus over $\cO_q(SL_2)$ is of the form $(V_n^\epsilon, \omega_n^\epsilon)$, with $n \in \NN$, $\epsilon \in \{-1,1\}$
and $(n,\epsilon) \neq (0,1)$.

Furthermore, each finite dimensional reduced differential calculus $(V,\omega)$
over $\cO_q(SL_2)$
can be decomposed into a direct sum:
\[(V,\omega) \cong \bigoplus\limits_{i=1}^d (V_{n_i}^{\epsilon_i}, \omega_{n_i}^{\epsilon_i}), \]
 where
$(n_1, \ldots n_d) \in \NN^d$, $(\epsilon_1, \ldots ,\epsilon_d) \in \{-1,1\}^d$
satisfies $(n_i, \epsilon_i) \neq (0,1)$ for all $i$ in $\{1, \ldots , d\}$ and $ (n_i,\epsilon_i) \neq (n_j, \epsilon_j)$  for all $i \neq j$.
\end{theoreme}

\begin{proof}
Since each finite dimensional reduced differential calculus over $\cO_q(SL_2)$ is inner,
and each simple finite dimensional Yetter-Drinfeld module over $\cO_q(SL_2)$
is of the form $V_{n,m}^\epsilon$, we conclude by Remark~\ref{rmq_simple} that the simple
finite dimensional reduced differential calculi over $\cO_q(SL_2)$ are the
$(V_n^\epsilon, \omega_n^\epsilon)$ with $(n,\epsilon) \neq (0,1)$.
Now if $(V,\omega)$ is a finite dimensional reduced differential calculus over $\cO_q(SL_2)$,
by~\cite{TAK}, we have an isomorphism of Yetter-Drinfeld modules
$V \cong \bigoplus\limits_{i=1}^d
V_i$ where each $V_i$ is a simple Yetter-Drinfeld module.
One then easily checks that for $i \in  \{1,\ldots, d\}$,  $\omega_i = \pi_i \circ \omega :  \cO_q(SL_2) \rightarrow V_i$
(where $\pi_i : V \rightarrow V_i$ is the canonical projection) is a reduced differential calculus.
We thus have $(V_i, \omega_i) \cong (V_{n_i}^{\epsilon_i}, \omega_{n_i}^{\epsilon_i})$
for some  $(n_i, \epsilon_i) \neq (0,1)$.
Then $(V, \omega) \cong \bigoplus\limits_{i=1}^d (V_{n_i}^{\epsilon_i}, \omega_{n_i}^{\epsilon_i})$, and by Lemma~\ref{lem_sum}, we have  $(n_i,\epsilon_i) \neq (n_j, \epsilon_j)$ when $i\neq j$.
\end{proof}

In order to give the classification of finite dimensional reduced differential calculi over free orthogonal Hopf algebras, we need the definition of the bilinear cogroupoid $\cB$. It will provide an explicit description of the equivalence between the categories of reduced differential calculi over a free
orthogonal Hopf algebra $\cB(E)$ and  $\cO_q(SL_2)$, for a well chosen $q$.

\begin{definition}
The \emph{bilinear cogroupoid} $\cB$ is defined as follows:
\begin{itemize}
\item $\ob(\cB) = \{E \in GL_n(\CC) \tq n \geqslant 1 \}$,
\item For $E, F \in \ob(\cB)$, and $m,n \geqslant 1$ such that
$E \in GL_m(\CC)$ and $F \in GL_n(\CC)$, $\cB(E,F)$ is the universal algebra generated by
elements $(a_{ij})_{\substack{1\leqslant i \leqslant m \\ 1 \leqslant j \leqslant n}}$
submitted to the relations:
\[F^{-1} a^t E a =I_n \quad \text{and} \quad
	a F^{-1}a^t E = I_m,
\]
where $a = (a_{ij})_{\substack{1\leqslant i \leqslant m \\ 1 \leqslant j \leqslant n}}$.
\item For $E,F,G \in \ob(\cB)$,
$\Delta_{E,F}^G : \cB(E,F) \rightarrow \cB(E,G) \otimes \cB(G,F)$,
$\varepsilon_E : \cB(E,E) \rightarrow \CC$ and 
$S_{E,F} : \cB(E,F)\rightarrow\cB(F,E)$ are characterized by:
\begin{align*}
\Delta_{E,F}^G(a_{ij}) & = \sum\limits_{k=1}^n a_{ik} \otimes a_{kj}, 
\text{where } n \geqslant 1 \text{ is  such that  } G \in GL_n(\CC), \\
\varepsilon_E (a_{ij}) & = \delta_{ij},\\
S_{E,F}(a_{ij}) & = (E^{-1}a^tF)_{ij}.
\end{align*}
\end{itemize}
For $E \in GL_n(\CC)$,  $\cB(E,E)$ is a Hopf algebra, which will also be denoted by $\cB(E)$,
and called the \emph{free orthogonal Hopf algebra associated with $E$}.
\end{definition}

\begin{rmq}
One easily checks that $\cO_q(SL_2) = \cB(E_q)$,
where
\[
E_q = 
\begin{pmatrix}
0 & 1 \\
-q^{-1} & 0
\end{pmatrix}.
\]
\end{rmq}

By~\cite[Corollary 3.5]{HG_co}, for $\lambda \in \CC$, the subcogroupoid $\cB^\lambda$ of
$\cB$ defined by
\[\cB^\lambda = \{ E \in GL_m(\CC) \tq n\geqslant 2, \tr(E^{-1}E^t) = \lambda\}\]
is connected (here ``$\tr$'' denotes the usual trace).

In the following, $E \in GL_m(\CC)$ with $m \geqslant 2$, denotes a matrix such that 
 any solution of the equation
$q^2 + \tr(E^{-1} E^t)q + 1 =0$ is not a root of unity.

If $q$ is a solution of this equation, we  have $\tr(E_q^{-1} E_q^t) = -q-q^{-1} = \tr(E^{-1}E^t)$, thus $E$ and $E_q$
are in the connected cogroupoid $\cB^\lambda$, where $\lambda =  -q-q^{-1}$.
The Hopf algebras $\cB(E_q) = \cO_q(SL_2)$ and $\cB(E)$ are thus
monoidally equivalent, and by Theorem~\ref{th_eq}, we have
an equivalence between the categories of
reduced differential calculi $\cR\cD\cC(\cO_q(SL_2))$ and $\cR\cD\cC(\cB(E))$
given by:
\[
	\appl[\cF_{E_q}^{E} :]{\cR\cD\cC(\cO_q(SL_2))}{\cR\cD\cC(\cB(E))}
{(V,\omega)}{(V \cotens\limits_{\cO_q(SL_2)} \cB(E_q, E), \overline{\omega}).}
\]

\begin{definition}
For $n$ in $\NN$ and $\epsilon \in \{-1,1\}$ such that $(n, \epsilon) \neq (0,1)$,
we denote by $W_n^\epsilon$ the $\cB(E)$-Yetter-Drinfeld module $V_n^\epsilon  \cotens\limits_{\cO_q(SL_2)} \cB(E_q, E)$.
We fix a non-zero right-coinvariant element
$\theta_n \in V_n \otimes V_n$
and we denote by $\eta_n^\epsilon : \cB(E) \rightarrow W_n^\epsilon$
the inner reduced differential calculus defined by
$\eta_n^\epsilon(x) = (\theta_n \otimes 1) \triangleleft x - \varepsilon(x)(\theta_n \otimes 1)$.
\end{definition}

By Remark~\ref{rmq_inner}, $\cF_{E_q}^{E}(V_n^\epsilon, \omega_n^\epsilon)$
is isomorphic to $(W_n^\epsilon, \eta_n^\epsilon)$ for all $n \in \NN$ and all
$\epsilon \in \{-1,1\}$ such that $(n, \epsilon) \neq (0,1)$.
According to Theorems~\ref{th_eq} and~\ref{th_class}, we obtain the following classification
of finite dimensional reduced differential calculi over $\cB(E)$.

\begin{proposition}
Each finite dimensional bicovariant differential calculus over $\cB(E)$ is inner.
\end{proposition}

\begin{theoreme}\label{th_class_free}
Every simple finite dimensional reduced differential calculus over $\cB(E)$ is of the form $(W_n^\epsilon , \eta_n^\epsilon)$, with $n \in \NN$, $\epsilon \in \{-1,1\}$ and $(n,\epsilon) \neq (0,1)$.

Furthermore, each finite dimensional reduced differential calculus $(W,\eta)$
over $\cB(E)$ can be decomposed into a direct sum:
\[(W,\eta) \cong \bigoplus\limits_{i=1}^d (W_{n_i}^{\epsilon_i}, \eta_{n_i}^{\epsilon_i}),
\]
 where
$(n_1, \ldots n_d) \in \NN^d$, $(\epsilon_1, \ldots ,\epsilon_d) \in \{-1,1\}^d$
satisfies $(n_i, \epsilon_i) \neq (0,1)$ for all $i$ in $\{1, \ldots , d\}$ and $ (n_i,\epsilon_i) \neq (n_j, \epsilon_j)$  for all $i \neq j$.
\end{theoreme}

\bibliographystyle{malpha}
\bibliography{biblio.bib}

\end{document}